\documentclass[10pt,reqno]{amsart}
\usepackage{amsmath}
\usepackage{amssymb}
\usepackage{amsthm}
\usepackage{verbatim}
\usepackage[pdftex]{graphics}
\usepackage{graphicx}
\usepackage{marvosym}
\usepackage{color}
\usepackage[table]{xcolor}
\usepackage{soul}

\newtheorem{theorem}{Theorem}
\newtheorem{lemma}[theorem]{Lemma}

\title[Invariant subspaces and Deddens algebras]{Invariant subspaces and Deddens algebras}
\author{Miguel Lacruz}
\address{M. Lacruz  \\ Facultad de Matem\'aticas, Universidad de Sevilla, Avenida Reina Mercedes, 41012 Seville (Spain)}
\email{lacruz@us.es}
\thanks{This research was partially supported by  Ministerio de Ministerio de Econom\'{\i}a y Competitividad  under grant   MTM \indent 2012-30748, and by Junta de Andaluc\'{\i}a under grant FQM-3737.}

\begin{document}
\begin{abstract}
It is shown that if the Deddens algebra \({\mathcal D}_T\) associated with a quasinilpotent operator \(T\) on a complex Banach  space is  closed and localizing then \(T\) has a nontrivial closed hyperinvariant subspace.
\end{abstract}

\date{\today}
%
%
\maketitle

\noindent
We shall represent by \({\mathcal B}(E)\)  the algebra of all bounded linear operators  on a complex Banach space \(E.\) 
Recall that the commutant of an operator \(T \in {\mathcal B}(E)\) is the subalgebra \(\{T\}^\prime \subseteq {\mathcal B}(E)\) of all operators that commute with \(T.\) 
A subspace \(F \subseteq E\) is said to be invariant under an operator \(T \in {\mathcal B}(E)\) provided that \(TF \subseteq F\). A subspace \(F \subseteq E\) is said to be  invariant under a subalgebra \({\mathcal R} \subseteq {\mathcal B}(E)\) provided that \(F\) is invariant under  every \(R \in {\mathcal R} \). A subspace \(F \subseteq E\) is said to be hyperinvariant under an operator \(T \in {\mathcal B}(E)\) provided that \(F\) is invariant under  the commutant \(\{T\}'\). A subalgebra \({\mathcal R} \subseteq {\mathcal B}(E)\) is said to be transitive provided that the only closed subspaces invariant under \({\mathcal R}\) are the trivial ones,  that is, \(F= \{0\}\) and \(F=E.\)  As it turns out,  this is equivalent to saying that for every \(x \in E \backslash \{0\},\) the subspace \(\{Rx \colon R \in {\mathcal R}\}\) is dense in \(E.\) 

Recall that an operator \(T \in \mathcal B(E)\) is said to be quasinilpotent provided that  \(\sigma(T)=\{0\}.\) According with the spectral radius formula,  \(T\) is quasinilpotent if and only if 
\begin{align}
r(T)=\lim_{n \to \infty} \|T^n\|^{1/n}=0.
\end{align}

Let \(T \in {\mathcal B}(E)\) and consider the {\em Deddens algebra} \(\mathcal D_T\)  associated with  \(T,\)  that is, the family of those operators \(X \in {\mathcal B}(E)\) for which there is a constant \(M>0\) such that for every \(n \in \mathbb N\) and for every \(f \in E,\) 
\begin{align}
\|T^n X f \| \leq M \|T^nf\|.
\end{align}
When \(T\) is invertible this is equivalent to saying that 
\begin{align}
\label{eqn:deddens}
\sup_{n \in \mathbb N} \|T^n X T^{-n} \| <\infty.
\end{align}
It is easy to see that \(\mathcal D_T\) is indeed a unital subalgebra of \({\mathcal B}(E)\) with the nice property that  \(\{T\}^\prime \subseteq \mathcal D_T.\) Also, \({\mathcal D_T}={\mathcal B}(E)\) in case \(T\) is an isometry.
These algebras are named after Deddens because he first introduced them in the  1970s in the context of nest algebras \cite{deddens}. The description of Deddens algebras associated with some special classes of operators has been recently obtained  by Petrovic \cite{petrovic2011a,petrovic2011b}.

Let \(T \in \mathcal B(E).\)  A complex scalar \(\lambda\) is said to be an {\em extended eigenvalue} for \(T\) provided  that there exists a nonzero operator \(X \in \mathcal B(E)\) such that \(TX= \lambda XT.\) Such an operator is called an {\em extended eigenoperator} for \(T\) corresponding to  the extended eigenvalue \(\lambda.\) These notions  became popular back in the  1970s when searching for invariant subspaces.  Recently, the concepts of extended eigenvalue and  extended eigenoperator have received  a considerable amount of attention, both in the context of  invariant subspaces \cite{LR,lambert} and in the study of extended eigenvalues and  extended eigenoperators for some special classes of operators \cite{BLP,BP, BS,lauric,petrovic2007,petrovic2011a}. 

Let \(\mathcal E_T(\lambda)\) denote the set of extended eigenoperators of \(T\) associated with an extended eigenvalue \(\lambda\)   and let \(\mathcal E_T\) denote the union of the sets \(\mathcal E_T(\lambda)\) when \(\lambda\) runs through all the extended eigenvalues  for \(T\) with \(|\lambda | \leq 1.\)
It is easy to see that \(\{T\}^\prime \subseteq \mathcal E_T \subseteq \mathcal D_T.\) Both  inclusions may be proper; for instance, Petrovic~\cite{petrovic2011a} showed that if \(W\) is an injective  unilateral shift on a Hilbert space then boths inclusions \(\{W\}^\prime \subset \mathcal E_W\) and  \( \mathcal E_W \subset \mathcal D_W\) are  proper. 

A subspace \({\mathcal X} \subseteq {\mathcal B}(E)\) is said to be {\em localizing} provided that there is a closed ball \(B \subseteq E\) such that \( 0 \notin B\) and such  that for every sequence of vectors \((f_n)\) in \(B\) there is a subsequence \((f_{n_j})\) and a sequence of operators \((X_j)\) in \({\mathcal X}\) such that \(\|X_j\| \leq 1\) and such that the sequence \((X_jf_{n_j})\) converges in norm to some nonzero vector. This notion  was introduced by Lomonosov, Radjavi, and Troitsky \cite{LRT}  as a side condition to build  invariant subspaces. A typical example of a localizing algebra is a subalgebra \({\mathcal R} \subseteq {\mathcal B}(E)\) such that the closure in the  weak operator topology  of the unit ball of \(\mathcal R\) contains a nonzero compact operator.

Rodr\'{\i}guez-Piazza and the author studied  some properties of localizing algebras  in a   recent paper~ \cite{LR}.   They also obtained a result on the existence of invariant subspaces that extends and unifies previous results of  Scott Brown \cite{brown} and   Kim, Moore and  Pearcy \cite{KMP}, on the one hand, and  Lomonosov, Radjavi and Troitsky \cite{LRT}, on the other hand. The result  goes as follows.

\begin{theorem}
\label{thm:LR}
Let \(T \in {\mathcal B}(E)\) be a nonzero operator, let \(\lambda \in {\mathbb C}\)  be an extended eigenvalue of \(T\)  such that the subspace \({\mathcal E_T(\lambda)}\) of all associated extended eigenoperators is localizing and suppose that either
\begin{enumerate}
\item \(|\lambda | \neq 1,\) or
\item \(|\lambda |=1\) and \(T\) is quasinilpotent.
\end{enumerate}
Then  \(T\) has a nontrivial closed hyperinvariant subspace.
\end{theorem}

\noindent
The aim of this note is to provide an extension of part (2) in Theorem \ref{thm:LR}  by replacing the assumption that the subspace \(\mathcal E_T(\lambda)\) be localizing with the  assumption that the Deddens algebra \(\mathcal D_T\)  be closed and localizing.  Our main result can be stated as follows.

\begin{theorem}
\label{thm:main}
Let \(T \in {\mathcal B}(E)\)  be a nonzero quasinilpotent operator. If the Deddens algebra \({\mathcal D}_T\) is closed and localizing then \(T\) has a nontrivial closed hyperinvariant subspace.
\end{theorem}

\noindent
Notice that, under the assumption that \(\mathcal D_T\) be closed, part (2) of Theorem \ref{thm:LR} is a  consequence of Theorem~\ref{thm:main} since \(\mathcal E_T(\lambda) \subseteq \mathcal D_T,\) and that Theorem~\ref{thm:main} is strictly more general than part (2)  of Theorem \ref{thm:LR},  because the inclusion \(\mathcal E_T \subseteq \mathcal D_T\) is proper in general.
Let us start with a  general  result  about Deddens algebras before we proceed with the proof of Theorem~\ref{thm:main}. This result characterizes when the Deddens algebra \(\mathcal D_T\) is closed in the operator norm. The corresponding result for spectral radius algebras was obtained by Lambert and Petrovic \cite{LP}.


\begin{lemma}
\label{thm:ubp}
 Let \(T \in \mathcal B(E).\) The following conditions are equivalent:
 \begin{enumerate}
 \item The Deddens algebra \(\mathcal D_T\) is closed in the operator norm topology.
 \item There is a constant \(M>0\) such that for every \(X \in \mathcal D_T,\) for all \(n \in \mathbb N\) and for all \(f \in E\) we have 
\begin{align}
\|T^n Xf\| \leq M \|X\| \cdot\|T^n f \| 
\end{align}
\end{enumerate}
\end{lemma}

\noindent 

\begin{proof}[Proof of Lemma \ref{thm:ubp}] Suppose  \(\mathcal D_T\) is closed and consider for every \(k \in  \mathbb N\) the closed set \(\mathcal F_k \) of those operators  \( X \in \mathcal D_T \) that satisfy the inequality \(  \|T^nXf \| \leq k \|T^nf \| \) for all  \(n \in \mathbb N\) and   for all \( f \in E.\) Then  we have
\[
\mathcal D_T = \bigcup_{k \in \mathbb N} \mathcal F_k.
\]
It follows from Baire's theorem that there is some \(k_0 \in \mathbb N\) such that \(\mathcal F_{k_0}\) has nonempty interior, that is, there is some \(X_0 \in \mathcal F_{k_0}\) and there is some \(\varepsilon >0\)  such that \(\{ X \in \mathcal D_T \colon \|X-X_0 \| \leq \varepsilon \} \subseteq \mathcal F_{k_0}.\) Let \(Y \in \mathcal D_T\) such that \(\|Y\| \leq 1\) and let \(X=X_0 + \varepsilon Y.\) Then we have
\(
\varepsilon \|T^nYf \| \leq \| T^n X_0f\| + \|T^nXf \| \leq 2k_0 \|T^nf\|.
\)
Finally, we conclude that for every \(X \in \mathcal D_T,\) for all \(n \in \mathbb N\) and for all \(f \in \mathcal D_T\) we have 
\[
\|T^nXf \| \leq \frac{2k_0}{\varepsilon} \|X\| \cdot  \|T^nf\|.
\]
The converse is trivial because  if such a constant \(M>0\) exists then 
\[
\mathcal D_T = \bigcap_{n \in \mathbb N} \bigcap_{f \in E} \{ X \in \mathcal B(E) \colon \|T^n Xf \| \leq M \|X \| \cdot \|T^nf\|\}.
\]
so that \(\mathcal D_T\)  is closed since it is the intersection of  a family of closed sets.
\end{proof}

\noindent
An easy proof of the nontrivial part of Lemma \ref{thm:ubp} can be obtained from the uniform boundedness principle in the special case that \(T\) is an invertible operator. The proof goes as follows.

 \begin{proof}[Proof of Lemma \ref{thm:ubp}  when \(T\) is invertible] Consider the operator \(\Phi \colon \mathcal D_T \to \mathcal D_T\) defined by the expression \(\Phi (X)=TXT^{-1}\) for all \(X \in \mathcal D_T.\)  Notice that \(\Phi^n(X)= T^nXT^{-n}\) for all \(n \in \mathbb N,\) so that \(\sup_n  \|\Phi^n(X)\| < \infty.\) Now it follows from the uniform boundedness principle that \(\sup_n \|\Phi_n\| <\infty.\) This means that there is a constant \(M>0\) such that \(\|T^n X T^{-n}\| \leq M \| X\| \) for all \(n \in \mathbb N\) and  for all \(X \in \mathcal D_T.\) Therefore  \(\|T^nXT^{-n}g \| \leq M \|X\| \cdot  \|g\|\) for all \(g \in E,\) and taking \(g=T^nf\) we get  \(\|T^nXf \| \leq M \|X\| \cdot \|T^nf\|.\)
\end{proof}

The key for the proof of Theorem \ref{thm:main} is a lemma that we have extracted from the proof of Theorem~2.3 in the paper of Lomonosov, Radjavi and Troitsky \cite{LRT}. This lemma can be stated as follows.

\begin{lemma} 
\label{thm:key}Let \(T \in \mathcal B(E)\) be a nonzero operator such that \(\{T\}^\prime\)  is a transitive  algebra,  let  \(\mathcal R \subseteq \mathcal B(E)\) be a localizing algebra such  that  \(\{T\}^\prime \subseteq \mathcal R,\) and let \(B \subseteq E\) be a closed ball that makes  \({\mathcal R}\) a localizing algebra. There is a constant  \(c>0\) such that for every \(f \in B\) there is an  \(X \in \mathcal R\) such that  \(TXf \in B\) and \(\|X\| \leq c.\)
\end{lemma}

\begin{proof}[Proof of Lemma \ref{thm:key}]
Assume the commutant \(\{T\}^\prime\) is a transitive algebra. Since the closed subspace \(\ker T\) is invariant under  \(\{T\}^\prime\) and  since \(T \neq 0,\) we must have \(\ker T=\{0\},\) so that \(T\) is injective. We ought to show  that there is  some constant  \(c>0\) such that for every \(f \in B\) there is an  \(X \in \mathcal R\) such that \(\|X\| \leq c\) and \(TXf \in B.\) We proceed by contradiction. Otherwise, for every \(n \in {\mathbb N}\) there is an \(f_n \in B\) such that the condition \(X \in \mathcal R\) and \(TXf_n \in B\) implies \(\|X\| >n.\) Since \(\mathcal R\) is localizing, there is a subsequence \((f_{n_j})\) and there is a sequence \((X_j)\) in \(\mathcal R\) with \(\|X_j\| \leq 1,\) and such that  \((X_jf_{n_j})\) converges in norm to some nonzero vector \(f \in E.\) Therefore,  \((TX_jf_{n_j})\) converges in norm to \(Tf.\)  Since \(T\) is injective,  \(Tf \neq 0,\) and since \(\{T\}^\prime\) is transitive, there is an \(R \in \{T\}^\prime\)  such that \(RTf \in {\rm int}\,B.\) Hence, there is some \(j_0 \geq 1\) such that \(RTX_j f_{n_j} \in B\) for all \(j \geq j_0.\) Since \(RT=TR,\) we have \(TRX_j f_{n_j} \in B\) for all \(j \geq j_0.\) Since  \( RX_j \in  \mathcal R,\)   the choice of the sequence \((f_n)\) implies  \(\| RX_j\| > n_j\)  for all \(j \geq j_0.\) Finally, this leads to a contradiction, because \(\|  RX_j\| \leq  \|R\|\) for all \(j \geq 1.\) This completes the proof of Lemma \ref{thm:key}.
\end{proof}

The technique for the  proof of Theorem \ref{thm:main} is an iterative procedure that is reminiscent of an argument at the end of the proof in Hilden's simplification for the striking theorem of Lomonosov \cite{lomonosov}  that a nonzero compact operator on a complex Banach space has a nontrivial hyperinvariant subspace. We recommend the    book of Rudin  \cite{rudin}  for an exposition of this  argument.

\begin{proof}[Proof of Theorem \ref{thm:main}]
Start with any vector \(f_0 \in B\) and use Lemma \ref{thm:key} to choose an operator \(X_1 \in {\mathcal D}_T\) such that \(\|X_1 \| \leq c\) and such that  \(T X_1f_0 \in B.\) Now use again Lemma \ref{thm:key}  to choose another operator  \(X_2 \in {\mathcal D}_T\) such that \(\|X_2\| \leq c\) and \(T X_2TX_1f_0 \in B.\) Continue this ping pong game to obtain a sequence of vectors \(f_n \in  B\) and a sequence of operators \(X_n \in  {\mathcal D}_T\) such that \(\|X_n\| \leq c\) and such that \(f_n = TX_n \cdots TX_1f_0.\)  Now apply Lemma \ref{thm:ubp} to find a constant \(M>0\) such that \(\|T^n Xf \| \leq M  \|X\| \cdot  \|T^nf \| \) for every \(X \in \mathcal D_T,\) for all \(n \in \mathbb N\) and for all \(f \in H.\) Notice that \(\|f_1 \|=\|TX_1f_0\| \leq  c M   \|Tf_0\|.\)  Also, notice that 
 \[
 \|f_2\|  =\|TX_2TX_1f_0\|  \leq c M \| T^2X_1f_0\|  \leq (cM)^2 \|T^2f_0\|,
 \]
 and in general \(\|f_n \| \leq (cM)^n   \|T^nf_0\|.\)  Let \(d=\min \{\|x\| \colon x \in B\}.\) It is plain that \(d >0\) because \(0 \notin B.\)  
 Then, for all \(n \in \mathbb N\) we have
\(
0< d \leq \|f_n\|  \leq (cM)^n \|T^n f_0\|,
\)
and this gives information on the spectral radius of \(T,\) namely,
\[
r(T)= \lim_{n \to \infty} \|T^n\|^{1/n} \geq \frac{1}{cM} >0.
\]
We arrived at a contradiction because \(T\) is quasinilpotent. This completes the proof of Theorem \ref{thm:main}.
\end{proof}

\end{document}